\documentclass[12pt]{article}
\usepackage{amsmath,amssymb,amsthm}
\usepackage{cases}
\usepackage{amsfonts}
\usepackage{color,xcolor,cite}
\usepackage[left=2.0cm,right=2.0cm,top=2.0cm,bottom=2.0cm]{geometry}
\usepackage[colorlinks,citecolor=blue,urlcolor=blue]{hyperref}

\numberwithin{equation}{section}
\newtheorem{theorem}{Theorem}[section]

\newtheorem{lemma}{Lemma}[section]

\newtheorem{definition}{Definition}[section]
\newtheorem{remark}{Remark}[section]

\newcommand{\bal}{\begin{align}}
\newcommand{\bbal}{\begin{align*}}
\newcommand{\beq}{\begin{equation}}
\newcommand{\eeq}{\end{equation}}
\newcommand{\bca}{\begin{cases}}
\newcommand{\eca}{\end{cases}}
\def\div{\mathord{{\rm div}}~}
\def\al{\alpha}
\def\a2{\alpha^2}
\def\w{\omega}
\def\S{\mathbf{S}_{t}}
\def\la{\varLambda^s}
\def\l1{\varLambda^{s-1}}
\newcommand{\pa}{\partial}
\newcommand{\fr}{\frac}
\newcommand{\na}{\nabla}
\newcommand{\De}{\Delta}

\newcommand{\cd}{\cdot}

\newcommand{\dd}{\mathrm{d}}

\newcommand{\R}{\mathbb{R}}

\newcommand{\M}{\mathcal{M}}
\newcommand{\N}{\mathcal{N}}

\newcommand{\D}{\Delta}
\newcommand{\nn}{\nonumber\\}

\newcommand{\g}{\big}

\begin{document}
\title{The uniform existence time and Zero-Alpha limit problem of the Euler-Poincaré equations}

\author{Min Li$^{1}$ and Zhaoyang Yin$^{2,3}$\footnote{Corresponding author's email: mcsyzy@mail.sysu.edu.cn} \\
\small $^1$ Department of Mathematics, Jiangxi University of Finance and Economics, Nanchang, 330032, China\\
\small $^2$ Department of Mathematics, Sun Yat-sen University, Guangzhou, 510275, China\\
\small $^3$ Faculty of Information Technology, Macau University of Science and Technology, Macau, China}
\date{}
%\date{\today}

\maketitle\noindent{\hrulefill}

{\bf Abstract:} We consider the Cauchy problem of the Euler-Poincaré equations in $\mathbb{R}^d$ with a varying dispersion parameter $\alpha$. Based on the convex entropy structure and the modified commutator estimates, we have proved that the Euler-Poincaré equations have a uniform existence time with respect to $\alpha$ in Sobolev spaces $H^s.$ Combined with the Bona-Simth method, we obtain convergence of the solutions to the Euler-Poincaré equations as $\alpha\to 0$ in the same space where the initial data are located. 


\vskip0mm\noindent{\hrulefill}

\section{Introduction}\label{sec1}
In this paper, we consider the Cauchy problem in $\R^d$ for Euler-Poincaré equations
\begin{equation}\label{E-P}\tag{EP}
\begin{cases}
\partial_tm+u\cdot \nabla m+\nabla u^T\cd m+(\mathrm{div} u)m=0, \qquad &(t,x)\in \R^+\times \R^d,\\
m=(1-\a2\De)u,\qquad &(t,x)\in \R^+\times \R^d,\\
u(0,x)=u_0,\qquad &x\in \R^d,
\end{cases}
\end{equation}
here $u,~m$ represent the velocity and momentum respectively, $\nabla u^T=\bigl(\pa_iu^j\bigr)$  is the transpose of the Jacobian $\nabla u=\bigl(\pa_ju^i\bigr)$, and $\al\in [0,\al_0],\al_0>0$ is a dispersion parameter with the length scale. To avoid any confusion, the first equation can be written in the component-wise form as
$$\pa_tm^i+\sum_{1\leq j\leq d}u^j\pa_j m^i+\sum_{1\leq j\leq d}\pa_iu^jm^j+\bigl(\sum_{1\leq j\leq d}\pa_ju^j\bigr)m^i=0, \qquad i=1,2,\cdots,d.$$
The equations \eqref{E-P} were first introduced by Holm, Marsden, and Ratiu in \cite{hmr1,hmr2} as a high dimensional generalization of the following Camassa-Holm equation for modeling and analyzing the nonlinear shallow water waves :
\begin{align}\label{che}\tag{CH}
  m_{t}+um_{x}+2u_{x}m=0, \ m=u-u_{xx}.
  \end{align}
Indeed, when $d=1$ the Euler-Poincar\'{e} equations are the same as the Camassa-Holm equation (\ref{che}). Also, the Euler-Poincaré equations were investigated as the system describe geodesic motion on the diffeomorphism group with respect to the kinetic energy norm in \cite{hs}. 

 For $d=1$, the equation (\ref{che}) was introduced by Camassa and Holm\cite{ch} as a bi-Hamiltonian model for shallow water waves. Most importantly, CH equation has peakon solutions of the form $Ce^{-|x-Ct|}$ which aroused a lot of interest in physics, see \cite{c5,t}. There is an extensive literature about the strong well-posedness, weak solutions and analytic or geometric properties of the CH equation, here we name some. Local well-posedness and ill-posedness for the Cauchy problem of the CH equation were investigated in \cite{ce2,d2,glmy}. Blow-up phenomena and global existence of strong solutions were discussed in \cite{c2,ce2,ce3,ce4}. The existence of global weak solutions and dissipative solutions were investigated in \cite{bc1,bc2,xz1}, more results can be found in the references therein.

 As part of the well-posedness theory, the continuity properties of the data-to-solution map is important. Recently, starting from the research of Himonas et al. \cite{hk,hkm}, the continuity properties of the data-to-solution maps of the Camassa-Holm type equations are gradually attracting interest of many authors, see \cite{lyz1,lyz2}. Most of the non-uniform constinuity results are established only on a bounded set near the origin. To overcome this limitation, Inci obtained a series of nowhere uniform continuity results including many Camassa-Holm type equations \cite{inc1,inc2}. And for the incompressible Euler equation, Bourgain and Li \cite{bl} showed that the data-to-solution map is nowhere-uniform continuity in $H^s(\R^d)$ with $s\geq 0$ by using an idea of localized Galilean boost. Based on Galilean boost method and the symmetric structure, we prove that the data-to-solution map of the Euler-Poincar\'{e} equations \eqref{E-P} is not uniformly continuous on any open subset $U\subset B^s_{p,r}(\R^d),s>\max\{1+\frac d2,\frac32\}$ \cite{ll}.

The first rigorous analysis of the Euler-Poincaré equations \eqref{E-P}  was done by Chae and Liu \cite{cli}, they eatablished the local existence of weak solution in $W^{2,p}(\R^d),\ p>d$ and local existence of unique classical solutions in $H^{s}(\R^d),\ s>\frac d 2+3$. Yan and Yin \cite{yy} further discussed the local existence and uniqueness of the solution to \eqref{E-P} in Besov spaces. On the other hand, Li, Yu and Zhai \cite{lyzz} proved that the solutions to \eqref{E-P} with a large class of smooth initial data blows up in finite time or exists globally in time, which settled an open problem raised by Chae and Liu \cite{cli}. Later, Luo and Yin have obtained a new blow-up result in the periodic case by using the rotational invariant properties of the equation\cite{luoy}. For more results of Euler-Poincar\'{e} equations, see \cite{luoy,zyl}.

The Zero-Alpha limit problem has been studied by Chae and Liu in their pioneering paper \cite{cli}. They prove that as the dispersion parameter vanishes, the weak solution converges
to a solution of the zero dispersion equation with sharp rate as $\al\to 0$, provided that
the limiting solution belongs to $C([0, T ); H^s (\R^d ))$ with $s > d/2 + 3$. This result has been extended to the modified two-component Euler–Poincaré equations by Yan and Liu \cite{yl}. 
 
It is worth to mention that, in Chae and Liu's work, the uniform existence time $T$ with respect to $\al$ is based on assumptions and has not been proven. In fact, by the standard energy method, the existence time $T_\al$ of the solution would shrink to zero as $\al\to 0$. On the other hand, in dealing with the convergence problem of the solutions of \eqref{E-P} as $\al\to 0$, previous results required the assumption that the space of initial data possesses better regularity than the space where convergence is established. Naturally, this leads to the emergence of the following

\vspace*{1em}
\noindent\textbf{Problem 1.} {\it For the solutions of Euler–Poincaré system \eqref{E-P} with the same initial data, do there exist uniform existence time $T>0$ with respect to $\al\in[0,\al_0]$? }

\noindent\textbf{Problem 2.} {\it As $\al\to 0,$ does the solution of \eqref{E-P} converge to the solution of zero dispersion equation ($\al=0$) in the same space where the initial data are located?}

\vspace*{1em}

For simplicity, we first transform \eqref{E-P} into a transport type system. According to Yan\cite{yy}, the Euler-Poincaré equations are equivalent to  
\begin{align}\label{lns}
&~~~~~~~~~~~~~~~~~~~~~~~~~~~~~\partial_tu+u\cdot \nabla u=-(I-\alpha^2\De)^{-1}\nonumber\\
&\Big(\alpha^2\ \div\big(\underbrace{\nabla u\nabla u+\nabla u\nabla u^T- \nabla u^T\nabla u-\nabla u(\div u)+\frac12(\nabla u : \nabla u)\mathbf{I}}_{\M(u,u)}\big)+\underbrace{u(\div u)+\nabla u^T u}_{\N(u,u)}\Big),
\end{align}
Based on the structure of $\M(u,u)$ and $\N(u,u)$, we define the following bilinear symmetric form 
\begin{equation}\label{qr}
  \begin{cases}
&\M(u,v)=\frac12\big(\nabla u\nabla v+\nabla u\nabla v^T-\nabla u^T\nabla v-\nabla u(\mathrm{div} v)+(\nabla u : \nabla v)\mathbf{I}\\
&\qquad\qquad\qquad\nabla v\nabla u+\nabla v\nabla u^T-\nabla v^T\nabla u-\nabla v(\mathrm{div} u)\big),\\
&\N(u,v)=\frac12\big(u~\div v+\nabla v^T u+ v~\div u+\nabla u^T v \big).
\end{cases}
\end{equation}
It should be noted that, $\M(u,v)$ is a bilinear matrix-valued function of $(\nabla u,\nabla v),$  and $\N(u,v)$ is a bilinear vector-valued function of $(u,\nabla u; v, \nabla v)$ in \eqref{qr}. And by intentionally constructed, they are symmetric on $(u, v):$ 
$$\M(u,v)=\M(v,u),\quad\text{and}\quad \N(u,v)=\N(v,u). $$
With these definitions in hand, for all $\al\in [0,\al_0]$ we can rewrite \eqref{E-P} to the nonlocal form
    \begin{equation}\label{epa}\tag{$EP_\alpha$}
      \begin{cases}
        \partial_tu+u\cdot \nabla u=-(I-\alpha^2\De)^{-1}\Big(\a2\div\M(u,u)+\N(u,u)\Big), \qquad &(t,x)\in \R^+\times \R^d,\\
      u(0,x)=u_0,\qquad &x\in \R^d.
      \end{cases}
      \end{equation}
Particularly, When $\alpha = 0$, the \eqref{epa} becomes 
\begin{equation}\label{ep0}\tag{$EP_0$}
  \begin{cases}
    \partial_tu+u\cdot \nabla u=-\N(u,u), \qquad &(t,x)\in \R^+\times \R^d,\\
  u(0,x)=u_0,\qquad &x\in \R^d,
  \end{cases}
  \end{equation}
which is also a symmetric hyperbolic system of conservation laws of the form 
$$\partial_tu+\div ( u\otimes u )+\frac12\nabla|u|^2=0$$ 
and possess a global convex entropy function $\frac12\partial_t|u|^2+\div (|u|^2 u)=0.$

Based on a newly discovered entropy convexity structure of the system \eqref{epa} and the Bona-Simth method, this paper will provide affirmative answers to the questions in \textbf{ Problem 1} and \textbf{ Problem 2}, namely the following two theorems.

  \begin{theorem}[{\bf Uniform time of existence}]\label{th1} Let $d\geq 2,s>1+\frac d2$ and $\alpha \in[0,\al_0]$. Assume that the initial data $u_0\in H^s(\mathbb{R}^d),$ and $\mathbf{S}_{t}^{\alpha}(u_0)$ be the solution of \eqref{epa} with initial data $u_0.$ There exists a time $T>0$ independent of $\al$ such that for all $\alpha \in[0,\al_0]$ there exists a solution $u(t)=\mathbf{S}_{t}^{\alpha}(u_0)$ of  \eqref{epa} bounded in $\mathcal{C}([0,T];H^s)$ independently of $\al.$ Moreover, the time existence $T$ depends only on $~ s,~ d$ and $\|u_0\|_{H^s}.$
  \end{theorem}
\begin{remark}
Note that, for the fixed $\al\in [0,\al_0],$ we already have the local well-posedness of the Cauchy problem \eqref{epa} in $H^s(\mathbb{R}^d),~s>1+\frac d2.$ To obtain the uniform existence time w.r.t. $\al$, by the standard continuity arguments, it is sufficient to establish $\frac{\dd }{\dd t}\|u(t)\|_{H^s}\leq C \|u(t)\|^2_{H^s}$ with some constant $C$ independent of  $\al$. For the first term in the RHS of \eqref{epa}, we can easily provide the uniform bound $\|\a2(I-\alpha^2\De)^{-1}\div\M(u,u)\|_{H^s}\leq C \|u(t)\|^2_{H^s}$. Due to the absence of dispersion factor $\a2$, we cannot obtain a similar estimate for the last term. In fact, the uniform inequality like $\|(I-\alpha^2\De)^{-1}\N(u,u)\|_{H^s}\leq C \|u(t)\|^2_{H^s}$ is not ture, this is the main difficulty that we need to overcome.   
\end{remark}
\begin{theorem}[{\bf Convergence in original space}]\label{th2} Let $d\geq 2,s>1+\frac d2$ and $\alpha \in(0,\al_0]$. Assume that the initial data $u_0\in H^s(\mathbb{R}^d),$ and $\mathbf{S}_{t}^{\alpha}(u_0),~\mathbf{S}_{t}^{0}(u_0)$ be the solutions of \eqref{epa} and \eqref{ep0} with initial data $u_0$ respectively. Then there exists a time $T=T(\|u_0\|_{H^s},s,d)>0$ such that  $\mathbf{S}_{t}^{\alpha}(u_0),\mathbf{S}_{t}^{0}(u_0)\in \mathcal{C}([0,T];H^s)$ and
  \begin{equation}\label{a20}
    \lim_{\alpha\to 0}\left\|\mathbf{S}_{t}^{\alpha}(u_0)-\mathbf{S}_{t}^{0}(u_0)\right\|_{L^\infty_TH^{s}}=0.   
  \end{equation}
  \end{theorem}
\begin{remark}
The previous convergence results of \eqref{E-P} as $\al\to 0$ required the initial value to have higher regularity than the space where convergence holds. Both in Chae-Liu \cite{cli} and Yan-Yin\cite{yl}, to achieve the convergence of the solutions in $L^\infty([0, T);H^{s})$, it is necessary to assume that the initial data $u_0$ at least belong to $H^{s+2}(\R^d)$. Theorem \ref{th2} does not require this additional assumption, that is to say the convergence holds in original space where the initial data are located.
\end{remark}

The remainder of this paper is organized as follows. In Section \ref{sec2}, we list some notations and recall basic results of the Littlewood-Paley theory. In Section \ref{sec3}, we present the proof of Theorem \ref{th1} and Theorem \ref{th2} by establishing some technical lemmas and propositions.

\section{Littlewood-Paley analysis}\label{sec2}
We first present some facts about the Littlewood-Paley decomposition, the nonhomogeneous Besov spaces and their some useful properties (see \cite{bcd} for more details).

Let $\mathcal{B}:=\{\xi\in\R^d:|\xi|\leq 4/3\}$ and $\mathcal{C}:=\{\xi\in\R^d:3/4\leq|\xi|\leq 8/3\}.$
Choose a radial, non-negative, smooth function $\chi:\R^d\mapsto [0,1]$ such that it is supported in $\mathcal{B}$ and $\chi\equiv1$ for $|\xi|\leq3/4$. Setting $\varphi(\xi):=\chi(\xi/2)-\chi(\xi)$, then we deduce that $\varphi$ is supported in $\mathcal{C}$. Moreover,
\begin{eqnarray*}
\chi(\xi)+\sum_{j\geq0}\varphi(2^{-j}\xi)=1 \quad \mbox{ for any } \xi\in \R^d.
\end{eqnarray*}
We should emphasize that the fact $\varphi(\xi)\equiv 1$ for $4/3\leq |\xi|\leq 3/2$ will be used in the sequel.

For every $u\in \mathcal{S'}(\mathbb{R}^d)$, the inhomogeneous dyadic blocks ${\Delta}_j$ are defined as follows
\begin{equation*}
{\Delta_ju=}
\begin{cases}
0,   &if \quad j\leq-2;\\
\chi(D)u=\mathcal{F}^{-1}(\chi \mathcal{F}u),  &if \quad j=-1;\\
\varphi(2^{-j}D)u=\mathcal{F}^{-1}\g(\varphi(2^{-j}\cdot)\mathcal{F}u\g), &if  \quad j\geq0.
\end{cases}
\end{equation*}
In the inhomogeneous case, the following Littlewood-Paley decomposition makes sense
$$
u=\sum_{j\geq-1}{\Delta}_ju\quad \text{for any}\;u\in \mathcal{S'}(\mathbb{R}^d).
$$
The nonhomogeneous low-frequency cut-off operator $S_n$ will be often used in the subsequent sections, which is defined by
\begin{align}
  S_n u=\sum_{j\leq n-1}{\Delta}_ju\quad \text{for any}\;u\in \mathcal{S'}(\mathbb{R}^d).\label{cut}
\end{align}
\begin{definition}\label{besov}
Let $s\in\mathbb{R}$ and $(p,r)\in[1, \infty]^2$. The nonhomogeneous Besov space $B^{s}_{p,r}(\R^d)$ is defined by
\begin{align*}
B^{s}_{p,r}(\R^d):=\Big\{f\in \mathcal{S}'(\R^d):\;\|f\|_{B^{s}_{p,r}(\mathbb{R}^d)}<\infty\Big\},
\end{align*}
where
\begin{numcases}{\|f\|_{B^{s}_{p,r}(\mathbb{R}^d)}=}
\left(\sum_{j\geq-1}2^{sjr}\|\Delta_jf\|^r_{L^p(\mathbb{R}^d)}\right)^{\fr1r}, &if $1\leq r<\infty$,\nonumber\\
\sup_{j\geq-1}2^{sj}\|\Delta_jf\|_{L^p(\mathbb{R}^d)}, &if $r=\infty$.\nonumber
\end{numcases}
\end{definition}
In particular, by the Fourier–Plancherel formula, Besov space $B^s_{2,2}(\R^d)$
coincides with the standard Sobolev space $H^s(\R^d).$ The following Bernstein's inequalities will be used in the sequel.
\begin{lemma} Let $\mathcal{B}$ be a Ball and $\mathcal{C}$ be an annulus. There exist constants $C>0$ such that for all $k\in \mathbb{N}\cup \{0\}$, any positive real number $\lambda$ and any function $f\in L^p(\R^d)$ with $1\leq p \leq q \leq \infty$, we have
\begin{align*}
&{\rm{supp}}\hat{f}\subset \lambda \mathcal{B}\;\Rightarrow\; \|D^kf\|_{L^q}:=\sup_{|\alpha^2|=k}\|\partial^\alpha f\|_{L^q}\leq C^{k+1}\lambda^{k+(\frac{d}{p}-\frac{d}{q})}\|f\|_{L^p},  \\
&{\rm{supp}}\hat{f}\subset \lambda \mathcal{C}\;\Rightarrow\; C^{-k-1}\lambda^k\|f\|_{L^p} \leq \|\D^kf\|_{L^p} \leq C^{k+1}\lambda^k\|f\|_{L^p}.
\end{align*}
\end{lemma}
\begin{lemma}[See \cite{bcd}]\label{inte}
Let $(s_1,s_2,p,r)\in \R^2\times [1,\infty]^2$, and $s_1<s_2,~0<\theta<1$, then we have
  \begin{align*}
    \|u\|_{B^{\theta s_1+(1-\theta)s_2}_{p,r}}\leq &\|u\|_{B^{s_1}_{p,r}}^\theta\|u\|_{B^{s_2}_{p,r}}^{1-\theta},\\
   \|u\|_{B^{\theta s_1+(1-\theta)s_2}_{p,1}}\leq &\fr{C}{s_2-s_1}\Bigl(\frac{1}{\theta}+\frac{1}{1-\theta}\Bigr)\|u\|_{B^{s_1}_{p,\infty}}^\theta\|u\|_{B^{s_2}_{p,\infty}}^{1-\theta}.
      \end{align*}
\end{lemma}
Then, we give some important product estimates which will be used throughout the paper.
\begin{lemma}[See \cite{bcd}]\label{pe}
  For $(p,r)\in[1, \infty]^2$ and $s>0$, $B^s_{p,r}(\R^d)\cap L^\infty(\R^d)$ is an algebra. Moreover, for any $u,v \in B^s_{p,r}(\R^d)\cap L^\infty(\R^d)$, we have
  \bbal
  &\|uv\|_{B^{s}_{p,r}}\leq C(\|u\|_{B^{s}_{p,r}}\|v\|_{L^\infty}+\|v\|_{B^{s}_{p,r}}\|u\|_{L^\infty}).
  \end{align*}
In addition, if $s>\max\big\{1+\frac d p,\frac32\big\}$, then
    \begin{align*}
    &\|uv\|_{B^{s-2}_{p,r}(\R^{d})}\leq C\|u\|_{B^{s-2}_{p,r}(\R^{d})}\|v\|_{B^{s-1}_{p,r}(\R^{d})}.
    \end{align*}
    \end{lemma}

We will use the following modified commutator estimates.
\begin{lemma}[{\bf Modified Commutator Estimates}]\label{mce}
  Let $\la_\alpha:=(1-\Delta)^{s/2}(1-\alpha^2\Delta)^{-1},$ and denote $\la=\la_0=(1-\Delta)^{s/2}.$ For any $s>0$ and $\na v \in L^\infty\cap H^{s-1}(\R^d),~u \in H^s(\R^d) $, we have
  \bbal
  &\|[\la_\alpha,v\cdot\nabla]u\|_{L^2}\leq C\big(\|\na v\|_{L^\infty}\|u\|_{H^s}+\|\na v\|_{H^{s-1}}\|\na u\|_{L^\infty}\big).
  \end{align*}
if $s< 1+\frac{d}2 $, and $\na v \in L^\infty\cap H^{d/2}(\R^d),~u \in H^s(\R^d) $, we also have
\bbal
&\|[\la_\alpha,v\cdot\nabla]u\|_{L^2}\leq C\|\na v\|_{L^\infty\cap H^{d/2}}\|u\|_{H^s}.
\end{align*}
\end{lemma}
\begin{proof}
Since the Bessel potential $K(x)=\mathcal{F}^{-1}\big((1+|\xi|^2)^{-1}\big)(x)\in L^1\cap L^q(\R^d)$ for any $1<q<\frac{d}{d-2}$ (or $1<q<\infty$ when $d=2$), see Li et al. \cite{lyzz}. We  then have a uniform bound for the operator $(1-\alpha^2\Delta)^{-1}$ in $L^p(\R^d),~1\leq p\leq\infty$ with respect to $\alpha\in[0,\al_0]$ as 
$$\|(1-\alpha^2\Delta)^{-1}u\|_{L^p}=\|\al^{-d}K(\frac{\cdot}\al)\ast u\|_{L^p}\leq \|\al^{-d}K(\frac{\cdot}\al)\|_{L^1}\|u\|_{L^p}=\|K\|_{L^1}\|u\|_{L^p}.$$
The rest of the proof follows directly from Lemma 2.100 of \cite{bcd}, we shall omit it here.
\end{proof}
\section{Proof of the main theorem }\label{sec3}

We first recall the local existence and uniqueness theory of solutions for the Cauchy problem \eqref{E-P} for fixed $\al=1$ in Besov spaces \cite{yy}, then provide some technical lemmas and propositions.
\subsection{Preparation and technical lemmas}
\begin{lemma}[Local well-posedness \cite{yy}]\label{hlocal}
Assume that 
\begin{align}\label{dprs}
  d\in \mathbb N_+, 1\leq p,r\leq\infty~ and ~s>\max\{1+\frac{d}{p},\frac{3}{2}\}. 
  \end{align}
  Let $u_{0}\in B^{s}_{p,r}(\mathbb{R}^{d})$, then there exists a time $T=T(\|u_{0}\|_{B^{s}_{p,r}(\mathbb{R}^{d})})>0$ such that for $\al=1$ the system \eqref{E-P} has a unique solution in
\begin{equation*}
  \begin{cases}
  C([0,T];B^{s}_{p,r}(\R^d))\cap C^1([0,T];B^{s-1}_{p,r}(\R^d)), &if~~ r<\infty,\\
  L^\infty([0,T];B^{s}_{p,\infty}(\R^d))\cap Lip([0,T];B^{s-1}_{p,\infty}(\R^d)), &if~~ r=\infty.
  \end{cases}
\end{equation*}
 And the mapping $u_0\mapsto u$ is continuous from $B^{s}_{p,r}(\R^d)$ into 
$C([0,T];B^{s'}_{p,r}(\R^d))\cap C^1([0,T];B^{s'-1}_{p,r}(\R^d))$
for all $s'<s$ if $r=\infty$, and $s'=s$ otherwise. Moreover, for all $t\in[0,T]$, there holds
\begin{align*}
\|u(t)\|_{B^{s}_{p,r}(\mathbb{R}^{d})}\leq C\|u_{0}\|_{B^{s}_{p,r}(\mathbb{R}^{d})}.
\end{align*}
\end{lemma}
Thanks to the modified commutator estimates provided in Lemma \ref{mce}, we can deduce the following global convexity lemma, which is a key observation of our proof and will be repeatedly refered throughout the paper.
\begin{lemma}[{\bf Global Convexity}]\label{gce}
  Let $~\N(u,v)=\frac12\big(u~\div v+\nabla v^T u+ v~\div u+\nabla u^T v \big)$ as defined in \eqref{qr}, and denote $\la_\alpha:=(1-\Delta)^{s/2}(1-\alpha^2\Delta)^{-1},$  $\la:=(1-\Delta)^{s/2}.$ Then, for any $s>0,$ and $\na u \in L^\infty\ (\R^d),~u \in H^s(\R^d) $, we have
  \begin{align}
  \big|\langle\la_\alpha\N(u,u),\la u\rangle_{L^2}\big|&\leq C\|\na u \|_{L^\infty}\|u\|^2_{H^s}.\label{gc1}\tag{i}
  \end{align}
More generally, if $~\na v \in L^\infty \cap H^{d/2} \cap H^s(\R^d)$ and $u \in H^s(\R^d) $, we have
  \begin{align}
    \big|\langle\la_\alpha\N(v,u),\la u\rangle_{L^2}\big|&\leq C\big(\|\na v\|_{L^\infty\cap H^{d/2}\cap H^{s-1}}\|u\|_{H^{s}}+\|\na v\|_{H^{s}}\|u\|_{L^\infty}\big)\|u\|_{H^{s}}.\label{gc2}\tag{ii}
  \end{align}
In case $s> 1+\frac{d}2 $, as $ H^{s-1}(\R^d)\hookrightarrow L^\infty\cap H^{d/2}(\R^d),$ then we have
\begin{align}
  \big|\langle\la_\alpha\N(u,u),\la u\rangle_{L^2}\big|&\leq  C\|u\|^3_{H^s},\label{gc3}\tag{iii}\\
  \big|\langle\la_\alpha\N(v,u),\la u\rangle_{L^2}\big|&\leq C\big(\|v\|_{H^{s}}\|u\|_{H^{s}}^2+\|v\|_{H^{s+1}}\|u\|_{H^{s-1}}\|u\|_{H^{s}}\big).\label{gc4}\tag{iv}
    \end{align}
\end{lemma}
\begin{proof}
  We shall only prove \eqref{gc1} here, since \eqref{gc2} and \eqref{gc1} are similar. Firstly, we denote $\la_\alpha=(1-\Delta)^{s/2}(1-\alpha^2\Delta)^{-1},$ by the definition of $\N(u,v)$ in \ref{qr}, we have
  \begin{align}
    &~~~~ \big|\langle\la_\alpha\N(u,u),\la u\rangle_{L^2}\big|\nonumber\\
    &=\Big|\sum_{1\leq i,j\leq d}  \int_{\R^d}\Big(\la_\alpha(u^i\partial_ju^j) \la u^i+\la(u^j\partial_iu^j) \la_\alpha u^i\Big)\dd x\Big|\nonumber\\
    &=\Big|\underbrace{\sum_{i,j}  \int\big(u^i\partial_j\la_\alpha u^j\la u^i+u^j\partial_i\la u^j \la_\alpha u^i \big)\dd x}_{I_a}+\underbrace{\sum_{i,j}\int\big([\la_\alpha,u^i\partial_j]u^j \la u^i+[\la,u^j\partial_i]u^j \la_\alpha u^i\big)\dd x}_{I_b}\Big|,\nonumber
  \end{align}
  exchange the summation index $i,j$ in the second term of $I_a$, we have 
  \begin{align}
    |I_a|&=\Big|\sum_{1\leq i,j\leq d}  \int_{\R^d}\Big(u^i\partial_j\la_\alpha u^j\la u^i+u^i\partial_j\la u^i \la_\alpha u^j\Big)\dd x\Big|\nonumber\\
    &=\Big|\sum_{1\leq i,j\leq d}  \int_{\R^d}u^i\partial_j\big(\la_\alpha u^j\la u^i\big)\dd x\Big|\nonumber\\
    &=\Big|\sum_{1\leq i,j\leq d}  -\int_{\R^d}\partial_ju^i\big(\la_\alpha u^j\la u^i\big)\dd x\Big|\nonumber\\
    &\leq C\|\na u\|_{L^\infty}\|\la_\alpha u\|_{L^2}\|\la u\|_{L^2}\leq C\|\na u\|_{L^\infty}\|u\|^2_{{H}^{s}},\label{use31}
  \end{align}
  for the part $I_b,$ using commutator estimate in Lemma \ref{mce} to get that
  \begin{align}
    |I_b| & \leq \sum_{i,j}\Big(\|[\la_\alpha,u^i\partial_j]u^j\|_{L^2}\|\la u^i\|_{L^2}+\|[\la,u^j\partial_i]u^j\|_{L^2} \|\la_\alpha u^i\|_{L^2}\Big)\nonumber\\
    &\leq C\|\na u \|_{L^\infty}\|\na u\|_{{H}^{s-1}}\|u\|_{{H}^{s}}.\label{use32}
  \end{align}
  By combining  \eqref{use31}\eqref{use32}, we directly obtain \eqref{gc1}. The proof of \eqref{gc2} can refer to the estimation of \eqref{whs4} in Theorem \ref{th2} below. \eqref{gc3} and \eqref{gc4} are corollaries of \eqref{gc1} and \eqref{gc2}.
  \end{proof}

With the Local well-posedness and the Global Convexity lemma in hands, we are now in a position to begin the proof of Theorem \ref{th1} .

\subsection{Proof of Theorem \ref{th1}}

In this section, we will prove the uniform bound w.r.t. $\alpha\in[0,\al_0]$ of the solution $\mathbf{S}_{t}^{\alpha}(u_0)$ in $H^s$. For fixed $\alpha>0$, by the classical local well-posedness result, we known that there exists a $T_\alpha=T(\|u_0\|_{H^s},s,\alpha)>0$ such that the Euler-Poincaré equations \eqref{epa} has a unique solution $\mathbf{S}_{t}^{\alpha}(u_0)\in\mathcal{C}([0,T_\alpha];H^s)$. 

We shall prove that there exists a uniform time $T=T(\|u_0\|_{H^s},s,d)>0$ independent of $\alpha$ such that $T\leq T_{\alpha}$ and  
\begin{align}\label{m1}
\|\mathbf{S}_{t}^{\alpha}(u_0)\|_{L_T^{\infty} H^s} \leq 2\left\|u_0\right\|_{H^{s}}, \quad \forall (\alpha,t) \in[0,\al_0]\times [0,T].
\end{align}
Moreover, if $u_0 \in H^{s'}$ for some $s'>s$, then there exists $C_s=C(\|u_0\|_{H^s},s',s,d)>0$ independent of $\alpha$ such that
\begin{align}\label{m2}
\|\mathbf{S}_{t}^{\alpha}(u_0)\|_{L_T^{\infty} H^{s'}}\leq C_s\left\|u_0\right\|_{H^{s'}}, \quad\forall (\alpha,t) \in[0,\al_0]\times [0,T].
\end{align}
%\leq e^{C\|u_0\|_{H^{s}}t}\left\|u_0\right\|_{H^{s'}}
To simplify notation, we set $u=\mathbf{S}_{t}^{\alpha}(u_0).$ Applying  $\la =(1-\De)^{s/2}$ to \eqref{epa} and take the $L^2$ inner product with $\varLambda^{s}u$, we obtain
\begin{align}
\frac12\frac{\dd }{\dd t}\|u\|^2_{{H}^{s}}=&-\int_{\R^d}\bigl(u\cd\na \la u\bigr)\cdot \la u\dd x-\int_{\R^d}[\la,u\cd\na] u\cdot \la u\dd x\label{uhs1}\\
&-\a2\int_{\R^d}\la(1-\alpha^2\Delta)^{-1}\big(\div \M(u,u) \big)\cdot \la u\dd x\label{uhs2}\\
&-\int_{\R^d}\la(1-\alpha^2\Delta)^{-1}\N(u,u)\cdot \la u\dd x\label{uhs3}
\end{align} 
To bound the first term \eqref{uhs1}, by the classical commutator estimation (in the case where $\alpha=0$ as stated in Lemma \ref{mce}), it is easy to obtain
\begin{align}
|\eqref{uhs1}|&\leq\frac12\Big|\int_{\R^d}\div u|\la u|^2\dd x\Big|+\Big|\int_{\R^d}[\la,u\cd\na] u\cdot \la u\dd x\Big|\nonumber\\
&\leq \frac12\|\div u\|_{L^\infty}\|u\|^2_{{H}^{s}}+\|[\la,u\cd\na] u\|_{L^2}\|\la u\|_{L^2}\nonumber\\
&\leq C\|\na u\|_{L^\infty}\|u\|^2_{{H}^{s}},\label{use1}
\end{align}
for the second term
\begin{align}\label{use2}
  |\eqref{uhs2}|&=\Big|\int_{\R^d}\Big(\frac{\a2(1-\Delta)}{1-\alpha^2\Delta}\varLambda^{s-2} \div \M(u,u) \Big)\cdot \la u\dd x\Big|\nonumber\\
  &\leq C\|\varLambda^{s-2} \div \M(u,u)\|_{L^2}\|\la u\|_{L^2}\nonumber\\
  &\leq C\|\M(u,u)\|_{{H}^{s-1}}\|u\|_{{H}^{s}}\nonumber\\
  &\leq C\|\na u\|_{L^\infty}\|\na u\|_{{H}^{s-1}}\|u\|_{{H}^{s}},
\end{align}
for the estimation of the last item \eqref{uhs3}, using \eqref{gc1} in Lemma \ref{gce} to get
\begin{align}\label{use3}
  |\eqref{uhs3}|&=\big|\langle\la_\alpha\N(u,u),\la u\rangle_{L^2}\big|\leq C\|\na u \|_{L^\infty}\|u\|^2_{H^s}.
\end{align}
Combining \eqref{use1} \eqref{use2} and \eqref{use3} yields
\begin{align}\label{sse}
\frac{\dd}{\dd t}\|u(t)\|_{H^s} \leq C\|\na u \|_{L^\infty}\|u\|_{H^s}.
\end{align}
When $s>1+\frac{d}{2}$, as $ H^s(\R^d)\hookrightarrow Lip\ (\R^d)$ we have
\begin{align*}
  \frac{\dd}{\dd t}\|u(t)\|_{H^s}\leq C\|u\|^2_{H^s},
  \end{align*}
by solving this ODE directly and using the continuity arguments, there exists a time $T=\frac{1}{2C\|u_0\|_{H^s}}$ such that
\begin{align}\label{sse1}
 \|u(t)\|_{H^s}\leq \frac{\|u_0\|_{H^s}}{1-C\|u_0\|_{H^s}t}\leq 2\|u_0\|_{H^s},\quad t\in [0,T].
  \end{align}
 Moreover, if $u_0 \in H^{s'}$ for some $s'>s>1+\frac{d}{2}$, as \eqref{sse} also true for $s'$, then for $t\in[0,T]$, using \eqref{sse1} and $\|\na u \|_{L^\infty}\leq C\|u(t)\|_{H^s}\leq 2C\|u_0\|_{H^s}$ we have  
\begin{align}\label{sse2}
  \frac{\dd}{\dd t}\|u(t)\|_{H^{s'}} \leq C\|u_0\|_{H^s}\|u\|_{H^{s'}},
  \end{align}
 with Gronwall's inequality, we obtain that
 \begin{align}\label{sse3}
  \|u(t)\|_{H^{s'}} \leq e^{C\|u_0\|_{H^{s}}t}\left\|u_0\right\|_{H^{s'}}\leq C_s\|u_0\|_{H^{s'}},\quad t\in [0,T].
  \end{align}
  We need to point out that all the constants $C,~C_s,~T$ mentioned above depend only on $(\|u_0\|_{H^{s}}, s, d)$, and specifically do not depend on $\alpha.$ This complete the proof of \eqref{m1} \eqref{m1} and the Theorem \ref{th1}.
\subsection{Proof of Theorem \ref{th2}}
Roughly speaking, our proof of Theorem \ref{th2} based on the decomposition
\begin{align}
&\qquad\S^\al (u_0)-\S^0(u_0)\nn
&=\underbrace{\S^\al (S_nu_0)-\S^0(S_nu_0)}_{\delta_{n}}+\underbrace{\S^\al (u_0)-\S^\al (S_nu_0)}_{\w_{n}^\alpha}-\underbrace{\big(\S^0(u_0)-\S^0(S_nu_0)\big)}_{\w_{n}^0},\label{app}
\end{align}
with some error terms $\w_{n}^\alpha, \w_{n}^0$. We will derive the smallness of $\w_{n}^\alpha, \w_{n}^0$ and $\delta_n$ in two separate steps. 

\vspace*{1em}
\noindent{\bf Step 1: Estimation of $\|\S^\al (u_0)-\S^\al (S_nu_0)\|_{H^s}$}.

In order to simplify the results, denoting that 
\begin{align*}
  u(t)=\S^{\alpha}(u_0),&\quad u_n(t)=\mathbf{S}_{t}^{\alpha}(S_nu_0),\\
  \w(t)=u(t)-u_n(t),&\quad\varpi(t) =u(t)+u_n(t)=\w(t)+2u_n(t),
\end{align*}
then $\w(t)$ satisfies $\w(0)=(\mathrm{Id}-S_n)u_0$ and
\begin{align}
  \partial_{t} \w=-u\cd\na \w-\w\cd\na u_n-\alpha^2\frac{\div}{1-\alpha^2\Delta} \M(\varpi,\w)-\frac{1}{1-\alpha^2\Delta}\N(\varpi,\w),\label{wa}
\end{align}
applying  $\la =(1-\De)^{s/2}$ to the bothside and take the $L^2$ inner product with $\la\w$, we obtain
\begin{align}
  \frac12\frac{\dd }{\dd t}\|\w\|^2_{{H}^{s}}=&-\int_{\R^d}\bigl(u\cd\na \la \w\bigr)\cdot \la \w\dd x-\int_{\R^d}[\la,u\cd\na] \w\cdot \la \w\dd x\label{whs1}\\
  &-\int_{\R^d}\la\bigl(\w\cd\na u_n\bigr)\cdot \la \w\dd x\label{whs2}\\
  &-\a2\int_{\R^d}\la(1-\alpha^2\Delta)^{-1}\big(\div \M(\varpi,\w) \big)\cdot \la \w\dd x\label{whs3}\\
  &-\int_{\R^d}\la(1-\alpha^2\Delta)^{-1}\N(\varpi,\w)\cdot \la \w\dd x\label{whs4}
  \end{align}
We would deduce the boundedness of \eqref{whs1}\eqref{whs2}\eqref{whs3} and \eqref{whs4} separately. For the first term \eqref{whs1}, integration by part yields
\begin{align}
  \big|\eqref{whs1}\big|&=\big|\frac12\int_{\R^d}\div u\big|\la\w\big|^2\dd x-\int_{\R^d}[\la,u\cd\na] \w\cdot \la \w\dd x\big|\nonumber\\
  &\leq \frac12\|\div u\|_{L^\infty}\|\la\w\|^2_{L^2}+C\big(\|\na u\|_{L^\infty\cap H^{d/2}}+\|\na u\|_{H^{s-1}}\big)\|\w\|_{H^s}\|\la\w\|_{L^2}\nonumber\\
  &\leq C\big(\|\na u\|_{L^\infty\cap H^{d/2}}+\|\na u\|_{H^{s-1}}\big)\|\w\|^2_{H^s}.\label{w11}
  \end{align}
Then, for \eqref{whs2}, the Moser type inequation and the  implies that
\begin{align}
  \big|\eqref{whs2}\big|
  &\leq C\big(\|\w\|_{H^s}\|\na u_n\|_{L^\infty}+\|\w\|_{L^\infty}\|\na u_n\|_{H^s}\big)\|\la\w\|_{L^2}\nonumber\\
  &\leq C\|\na u_n\|_{L^\infty}\|\w\|^2_{H^s}+C\|u_n\|_{H^{s+1}}\|\w\|_{L^\infty}\|\w\|_{H^s}.\label{w22}
  \end{align}
For the third term \eqref{whs3}, through direct calculation we have
\begin{align}
  \big|\eqref{whs3}\big|&=\Big|\int\frac{\a2(1-\Delta)}{1-\alpha^2\Delta}\varLambda^{s-2}\big(\div \M(\varpi,\w) \big)\cdot \la \w\dd x\Big|\nonumber\\
  &\leq C\|\varLambda^{s-1} \M(\varpi,\w)\|_{L^2}\|\la\w\|_{L^2}\nonumber\\
  &\leq C\|\na \varpi\|_{H^{s-1}}\|\na \w\|_{H^{s-1}}\|\w\|_{H^s}\nonumber\\
  &\leq C\|\na \varpi\|_{H^{s-1}}\|\w\|^2_{H^s}.\label{w33}
  \end{align}
For the last term \eqref{whs4}, we first split it into the following parts 
\begin{align}
	&~\quad\eqref{whs4}\nn
  &=\int_{\R^d}\la(1-\alpha^2\Delta)^{-1}\N(\w,\w)\cdot \la \w\dd x+\int_{\R^d}\la(1-\alpha^2\Delta)^{-1}\N(2 u_n,\w)\cdot \la \w\dd x\nn
  &=\int_{\R^d}\la_\alpha\N(\w,\w)\cdot \la \w\dd x 
	+\sum_{i,j}\int\la_\alpha(\w^i\partial_j u_n^j+\w^j\partial_i u_n^j+u_n^i\partial_j\w^j+ u_n^j\partial_i\w^j)\la \w^i\dd x\nn
  &=\langle \la_\alpha\N(\w,\w), \la \w\rangle_{L^2}
	+\underbrace{\sum_{i,j}\int\la_\alpha(\w^i\partial_j u_n^j+\w^j\partial_i u_n^j)\la \w^i\dd x}_{\mathrm{II}_0}+\nn
	&\underbrace{\sum_{i,j}\int( u_n^i\partial_j\la_\alpha\w^j\la\w^i+ u_n^j\partial_i\la\w^j \la_\alpha \w^i)\dd x}_{\mathrm{II}_a}+\underbrace{\sum_{i,j}\int\big([\la_\alpha, u_n^i\partial_j]\w^j \la \w^i+[\la, u_n^j\partial_i]\w^j \la_\alpha \w^i)\dd x}_{\mathrm{II}_b},\nonumber
  \end{align}
  here we have used the notation $\la_\alpha=(1-\Delta)^{s/2}(1-\alpha^2\Delta)^{-1}.$ By the global convexity of $\N$ in Lemma \ref{gce}, we have
  \begin{align}
   \big|\langle \la_\alpha\N(\w,\w), \la \w\rangle_{L^2}\big|\leq C\|\na \w\|_{L^\infty}\|\w\|^2_{H^s}\label{ii}
  \end{align}
  For $\mathrm{II}_0$, the Moser inequality implies that
  \begin{align}
    \big|\mathrm{II}_0\big|
    &\leq C\big(\|\w\|_{H^s}\|\na u_n\|_{L^\infty}+\|\w\|_{L^\infty}\|\na u_n\|_{H^s}\big)\|\la\w\|_{L^2}\nonumber\\
    &\leq C\|\na u_n\|_{L^\infty}\|\w\|^2_{H^s}+C\|u_n\|_{H^{s+1}}\|\w\|_{L^\infty}\|\w\|_{H^s}.\label{ii0}
    \end{align}
  Inspired by the previous estimate of \eqref{uhs3}, we exchange the summation index $i,j$ in the second term of $\mathrm{II}_a$ to get that
  \begin{align}
    \big|\mathrm{II}_a\big|&=\Big|\sum_{i,j}\int( u_n^i\partial_j\la_\alpha\w^j\la\w^i+ u_n^i\partial_j\la\w^i \la_\alpha \w^j)\dd x\Big|\nonumber\\
    &=\Big|\sum_{i,j}\int u_n^i\partial_j\big(\la_\alpha\w^j\la\w^i\big)\dd x\Big|\nonumber\\
    &=\Big|-\sum_{i,j}\int\partial_j u_n^i\la_\alpha\w^j\la\w^i\dd x\Big|\nonumber\\
    &\leq C\|\na u_n\|_{L^\infty}\|w\|_{H^s}^2\label{iia}
    \end{align}
    Thanks for the modified commutator estimate in Lemma \ref{mce}, we can bound $\mathrm{II}_b$ by 
    \begin{align}
    \big|\mathrm{II}_b\big|&=\Big|\sum_{i,j}\int\big([\la_\alpha, u_n^i\partial_j]\w^j \la \w^i+[\la, u_n^j\partial_i]\w^j \la_\alpha \w^i)\dd x\Big|\nonumber\\
    &\leq\sum_{i,j}\Big(\|[\la_\alpha, u_n^i\partial_j]\w^j\|_{L^2} \|\la \w^i\|_{L^2}+\|[\la, u_n^j\partial_i]\w^j\|_{L^2} \|\la_\alpha \w^i\|_{L^2}\Big)\nonumber\\
    &\leq C\Big(\|\na u_n\|_{L^\infty}\|\w\|_{H^s}+\|\na u_n\|_{H^{s-1}}\|\na \w\|_{L^\infty}\Big)\|\w\|_{H^s}\label{iib}
    \end{align}
along with \eqref{ii}\eqref{ii0} and \eqref{iia}, we conclude that
\begin{align}
  \big|\eqref{whs4}\big| &\leq C\|\w\|^2_{H^s}+C\|u_n\|_{H^{s+1}}\|\w\|_{L^\infty}\|w\|_{H^s}.\label{w44}
  \end{align}
Combining \eqref{w11} \eqref{w22} \eqref{w33} and \eqref{w44} to obtain that
\begin{align}
  \frac{\dd }{\dd t}\|\w\|_{H^s} &\leq C\|\w\|_{H^s}+C\|u_n\|_{H^{s+1}}\|\w\|_{L^\infty}\nn
  &\leq C\|\w\|_{H^s}+C2^n\|\w\|_{H^{s-1}},\label{half}
  \end{align}
here we have used the embedding $H^{s-1}(\R^d)\hookrightarrow L^\infty$ when $s>1+\frac{d}2,$ and the fact that $\|u_n\|_{H^{s+1}}\leq C\|S_n u_0\|_{H^{s+1}}\leq C2^n\|u_0\|_{H^{s}}.$ 

To close \eqref{half}, we have to estimate $\|\w\|_{H^{s-1}},$ applying  $\l1=(1-\De)^{(s-1)/2}$ to the bothside of \eqref{wa} again, and take the $L^2$ inner product with $\l1\w,$ then
\begin{align}
  &\frac12\frac{\dd }{\dd t}\|\w\|^2_{{H}^{s-1}}\label{s-1}\\
  =&-\int_{\R^d}\bigl(u\cd\na \l1 \w\bigr)\cdot \l1 \w\dd x-\int_{\R^d}[\l1,u\cd\na] \w\cdot \l1 \w\dd x\nonumber\\
  &-\int_{\R^d}\l1\bigl(\w\cd\na v\bigr)\cdot \l1 \w\dd x-\a2\int_{\R^d}\l1(1-\alpha^2\Delta)^{-1}\big(\div \M(\varpi,\w) \big)\cdot \l1 \w\dd x\nonumber\\
  &-\int_{\R^d}\l1(1-\alpha^2\Delta)^{-1}\N(\w,\w)\cdot \l1 \w\dd x-\int_{\R^d}\l1(1-\alpha^2\Delta)^{-1}\N(2u_n,\w)\cdot \l1 \w\dd x.\nonumber
\end{align}
By following the same procedure as the $H^{s}$ estimates in \eqref{w11}\eqref{w22} \eqref{w33} and \eqref{w44}, we can establish the bounds of \eqref{s-1} line by line:
\begin{align}
  \frac{\dd }{\dd t}\|\w\|^2_{{H}^{s-1}}\leq C\Bigl(&\big(\|\na u\|_{L^\infty\cap H^{d/2}}+\|\na u\|_{H^{s-2}}\big)\|\w\|^2_{H^{s-1}}\nn
  &+\|\na u_n\|_{L^\infty}\|\w\|^2_{H^{s-1}}+\|u_n\|_{H^{s}}\|\w\|_{L^\infty}\|\w\|_{H^{s-1}}+\|\na \varpi\|_{H^{s-1}}\|\w\|^2_{H^{s-1}}\nn
  &+\|\w\|^2_{H^{s-1}}+\|u_n\|_{H^{s}}\|\w\|_{L^\infty}\|w\|_{H^{s-1}}\Bigr)\nonumber\\
\leq C &\big(\|u\|_{H^{s}}+\|u_n\|_{H^{s}}\big)\|\w\|^2_{H^{s-1}}\leq C_{\textit{\tiny T}}\|\w\|^2_{H^{s-1}}
\end{align}
Then, by Gronwall's inequality
\begin{align}
  \|\w\|_{{H}^{s-1}}&\leq C_{\textit{\tiny T}}\|\w_0\|_{{H}^{s-1}}=C\|(\mathrm{Id}-S_n)u_0\|_{{H}^{s-1}}\leq C2^{-n}\|(\mathrm{Id}-S_n)u_0\|_{{H}^{s}}.\label{ws-1}
\end{align}
Plug \eqref{ws-1} into \eqref{half}, we arrive at 
\begin{align}
  \frac{\dd }{\dd t}\|\w\|_{H^s} &\leq C\|\w\|_{H^s}+C_1\|\w_0\|_{H^s}.\nonumber
  \end{align}
Using Gronwall's inequality again, we get  
\begin{align}
 \|\w(t)\|_{H^s} &\leq Ce^{Ct}\|\w_0\|_{H^s}\leq C_{\textit{\tiny T}}\|(\mathrm{Id}-S_n)u_0\|_{{H}^{s}},\qquad \forall t\in[0,T].\label{w9}
  \end{align}

  \vspace*{1em}
  \noindent{\bf Step 2: Estimation of $\|\mathbf{S}_{t}^{\alpha}(S_nu_0)-\mathbf{S}_{t}^{0}(S_nu_0)\|_{H^s}$}.

For simplicity, denote that 
$$\delta(t)=u_n(t)-v_n(t)\quad\text{while}\quad u_n(t)=\mathbf{S}_{t}^{\alpha}(S_nu_0), ~v_n(t)=\mathbf{S}_{t}^0(S_nu_0).$$
We can verify that $\delta(t)$ satisfies $\delta(0)=0$ and the following equation
\begin{align}
  &~\quad\partial_{t} \delta+u_n\cd\na \delta+\delta\cd\na v_n\nonumber\\
  &=\frac{-\alpha^2\div}{1-\alpha^2\Delta} \M(u_n,u_n)-\frac{1}{1-\alpha^2\Delta}\N(u_n,u_n)+\N(v_n,v_n)\nonumber\\
  &=\frac{-\alpha^2\div}{1-\alpha^2\Delta} \M(u_n,u_n)-\frac{\alpha^2\Delta}{1-\alpha^2\Delta}\N(v_n,v_n)-\frac{1}{1-\alpha^2\Delta}\N(u_n+v_n,\delta)\nn
  &=\frac{-\alpha^2\div}{1-\alpha^2\Delta}\Bigl( \M(u_n,u_n)+\Delta\N(v_n,v_n)\Bigr)-\frac{1}{1-\alpha^2\Delta}\Bigl(\N(\delta ,\delta)+\N(2v_n ,\delta)\Bigr).\label{da}
\end{align}
Applying  the operator $\la =(1-\De)^{s/2}$ and take the $L^2$ inner product with $\la\delta $ which leads to  
\begin{align}
  &~\quad\frac12\frac{\dd }{\dd t}\|\delta \|^2_{{H}^{s}}\nonumber\\
  =&-\int_{\R^d}\bigl(u_n\cd\na \la \delta \bigr)\cdot \la \delta \dd x-\int_{\R^d}[\la,u_n\cd\na] \delta \cdot \la \delta \dd x-\int_{\R^d}\la\bigl(\delta \cd\na v_n\bigr)\cdot \la \delta \dd x\nonumber\\
  &-\a2\int_{\R^d}\la(1-\alpha^2\Delta)^{-1}\big(\div \M(u_n,u_n ) +\Delta \N(v_n,v_n )\big)\cdot \la \delta \dd x\nonumber\\
  &-\int_{\R^d}\la(1-\alpha^2\Delta)^{-1}\N(\delta ,\delta)\cdot \la \delta\dd x-\int_{\R^d}\la(1-\alpha^2\Delta)^{-1}\N(2v_n ,\delta)\cdot\la \delta\dd x\nonumber\\
  \leq&\quad\|\div u_n\|_{L^\infty}\|\la \delta\|^2_{L^2}+\|\na u_n\|_{H^{s-1}\cap L^\infty}\|\delta\|_{H^{s}}^2+\|\na v_n\|_{H^{s}\cap L^\infty}\|\delta\|_{H^{s}}^2\nn
  &\quad+\a2 C\|\div \M(u_n,u_n ) +\Delta \N(v_n,v_n )\|_{H^{s}}\|\delta\|_{H^{s}}\nn
  &\quad+C\bigl(\|\delta\|_{H^{s}}^3+\|\na v_n\|_{L^\infty\cap H^{s-1}}\|\delta\|_{H^{s}}^2+\|\na v_n\|_{H^{s}}\|\delta\|_{L^\infty}\|\delta\|_{H^{s}}\bigr)\nn
  \leq &\quad C\big(\|u_n \|_{H^{s}}+\|v_n \|_{H^{s+1}}\big)\|\delta \|_{H^{s}}^2+\a2\big(\|u_n \|_{H^{s+2}}+\|v_n \|_{H^{s+3}}\big)\|\delta\|_{H^{s}}\label{de2}
  \end{align}
As indicated by the growth rate $\|u_n, v_n\|_{H^{s+k}}\leq C_{\textit{\tiny T}}\|S_nu_0\|_{H^{s+k}}\leq C2^{kn},$ it can be inferred from \eqref{de2} that
\begin{align*}
 \frac{\dd }{\dd t}\|\delta \|_{{H}^{s}}\leq &\quad C2^n\|\delta \|_{H^{s}}+C\a2 2^{3n}
  \end{align*}
and the Gronwall's inequality implies
\begin{align}
  \|\delta (t)\|_{H^s} &\leq C\a2 2^{3n}te^{C2^nt}\leq C\a2 2^{3n}e^{C2^n},\qquad \forall t\in[0,T].\label{d9}
   \end{align}
Back to the decomposition \eqref{app}, with the combining results of \eqref{w9} and \eqref{d9}, we obtain that 
\begin{align}
   \|\S^\al (u_0)-\S^0(u_0)\|_{H^s} \leq C\a2 2^{3n}e^{C2^n}+C\|(\mathrm{Id}-S_n)u_0\|_{{H}^{s}},\qquad \forall (t,n)\in [0,T]\times \mathbb{N}^+ .\label{u9}
\end{align}
From \eqref{u9}, we can conclude the theory by such that: given $u_0\in H^s(\R^d)$ and any $\epsilon>0,$ we first find a large $N$ such that $C\|(\mathrm{Id}-S_N)u_0\|_{{H}^{s}}< \epsilon/2.$ Then, for the fixed $N$, we could choose a small $\alpha_{\textit{\tiny N}}$ such that $C\alpha^2_{\textit{\tiny N}} 2^{3\textit{\tiny N}}\exp (C2^\textit{\tiny N})< \epsilon/2,$ thus $\|\S^{\alpha_{\textit{\tiny N}}} (u_0)-\S^0(u_0)\|_{H^s}< \epsilon.$ As a result, we obtain
$$\lim_{\al\to 0} \|\S^\al (u_0)-\S^0(u_0)\|_{H^s}=0,\qquad \forall t\in [0,T],$$
this complete the proof of Theorem \ref*{th2}.

\vspace*{1em}
\noindent\textbf{Acknowledgements.}
M. Li was partially supported by the Jiangxi Provincial Natural Science Foundation (No.20232BAB201013 and 20212BAB201008), and China Postdoctoral Science Foundation (No. 2023M731431), and the National Natural Science Foundation of China (No.12101271). 

\vspace*{1em}
\noindent\textbf{Conflict of interest}
The authors declare that they have no conflict of interest.

\end{document}